\documentclass[11pt]{amsart}
\usepackage[T1]{fontenc}
\usepackage{lmodern}
\usepackage{amsmath,amssymb,mathtools}
\usepackage{microtype}
\usepackage{needspace}
\usepackage[margin=1in]{geometry}
\usepackage[hidelinks]{hyperref}
\newtheorem{theorem}{Theorem}[section]
\newtheorem{proposition}[theorem]{Proposition}
\newtheorem{lemma}[theorem]{Lemma}
\newtheorem{corollary}[theorem]{Corollary}
\theoremstyle{remark}
\newtheorem{remark}[theorem]{Remark}
\newcommand{\N}{\mathbb N}
\newcommand{\Q}{\mathbb Q}
\newcommand{\R}{\mathbb R}
\newcommand{\B}{\mathcal B}
\newcommand{\A}{\mathcal A}
\newcommand{\C}{\mathcal C}
\newcommand{\Aut}{\operatorname{Aut}}
\newcommand{\End}{\operatorname{End}}
\newcommand{\Epi}{\operatorname{Epi}}
\newcommand{\Hom}{\operatorname{Hom}}
\newcommand{\Inn}{\operatorname{Inn}}
\newcommand{\Subg}{\operatorname{Subg}}
\newcommand{\Pro}{\operatorname{Pro}_{\omega}}
\newcommand{\GrpEpi}{\mathsf{GrpEpi}}
\newcommand{\mGraph}{\mathsf{mGraph}}
\newcommand{\mGraphEpi}{\mathsf{mGraphEpi}}
\newcommand{\Ebd}{E_{\mathrm{bd}}}
\newcommand{\perm}{\mathrm{perm}}
\newcommand{\topiso}{\cong_{\mathrm{top}}}
\newcommand{\pciso}{\cong_{\mathrm{pc}}}
\newcommand{\naiso}{\cong_{\mathrm{NA}}}
\newcommand{\PC}{\mathsf{PC}}
\newcommand{\restr}{\mathbin{\upharpoonright}}
\newcommand{\norm}[1]{\left\lVert #1\right\rVert}
\newcommand{\doi}[1]{\href{https://doi.org/#1}{\nolinkurl{#1}}}
\title[Isomorphism of procountable groups]{Topological isomorphism of procountable groups is universal analytic}

\author{Xinan Dai}
\address{College of Future Information Technology,
Fudan University, Shanghai, China}
\curraddr{Department of Artificial Intelligence,\newline
School of Engineering, Westlake University, Hangzhou, China}
\email{xndai23@m.fudan.edu.cn}

\author{Wenhao Deng}
\address{University of Glasgow, Glasgow, United Kingdom}
\curraddr{Department of Artificial Intelligence,\newline
School of Engineering, Westlake University, Hangzhou, China}
\email{dengwenhao@westlake.edu.cn}

\author{Yingdong Shi}
\address{School of Information Science and Technology,
ShanghaiTech University, Shanghai, China}
\email{shiyd2023@shanghaitech.edu.cn}

\author{Tailin Wu}
\address{Department of Artificial Intelligence,\newline
School of Engineering, Westlake University, Hangzhou, China}
\email{wutailin@westlake.edu.cn}

\author{Yuchen Yang}
\address{Department of Artificial Intelligence,\newline
School of Engineering, Westlake University, Hangzhou, China}
\email{yangyuchen@westlake.edu.cn}

\date{\today}

\hypersetup{pdftitle={Topological isomorphism of procountable groups is universal analytic},pdfsubject={Descriptive set theory and topological groups}}
\begin{document}
\raggedbottom
\begin{abstract}
We prove that topological isomorphism of procountable groups is a universal analytic equivalence relation, answering a question of Gao, Nies, and Paolini. The same conclusion follows for non-Archimedean Polish groups. More strongly, there is one countable group $H$ for which universality already holds among inverse limits of sequences of surjective endomorphisms of $H$. Thus all the classification complexity can be carried by the bonding maps. We encode permutative equivalence of unconditional basic sequences by integer weights with prescribed symmetries. An extension of Gao, Nies, and Paolini's tree construction recovers bounded weight differences from uniform continuity while allowing these symmetries to act. Prze\'zdziecki's almost-full functor transfers the resulting inverse graph systems to groups. Identifying the stage graphs with a single graph makes the passage from weights to bonding endomorphisms continuous.
\end{abstract}
\maketitle

\section{Introduction}

A topological group is \emph{procountable} if it is topologically isomorphic to the inverse limit of a sequence of countable discrete groups with surjective bonding homomorphisms \cite[Definition~1.1]{GNP}. This class contains both countable discrete groups and metrizable profinite groups. Passing from these subclasses to general inverse limits changes the complexity of the isomorphism problem. Our purpose is to determine that complexity using Borel reducibility.

Write $\N=\{0,1,2,\ldots\}$ and $\N_{>0}=\{1,2,\ldots\}$. A non-Archimedean Polish group is a separable, completely metrizable topological group whose identity has a neighborhood basis of open subgroups. Becker and Kechris proved that these are precisely the closed subgroups of $S_\infty=\operatorname{Sym}(\N)$, up to topological isomorphism, where $S_\infty$ has the pointwise-convergence topology \cite[Theorem~1.5.1]{BK}. A standard Borel space is a set equipped with the Borel $\sigma$-algebra of a Polish topology. The Effros Borel space $\mathcal F(S_\infty)$ of closed subsets is standard Borel \cite[Section~12.C]{Kechris}, and its subset $\Subg(S_\infty)$ of closed subgroups is Borel \cite[Lemma~2.5]{KNT}. We use this parameter space for non-Archimedean Polish groups, and write $\PC$ for its subclass of procountable groups.

Borel reducibility makes the classification problem precise. For equivalence relations $E$ and $F$ on standard Borel spaces $X$ and $Y$, we write $E\leq_B F$ if there is a Borel map $f:X\to Y$ such that
\begin{equation}\label{eq:borel-reduction}
 x\mathrel E x'\quad\Longleftrightarrow\quad f(x)\mathrel F f(x')
 \qquad(x,x'\in X).
\end{equation}
If both $E\leq_BF$ and $F\leq_BE$, then $E$ and $F$ are \emph{Borel bireducible}. A subset of a standard Borel space $X$ is \emph{analytic} if it is the projection of a Borel subset of $X\times\N^\N$; a relation is analytic if its graph is analytic. An equivalence relation is \emph{classifiable by countable structures} if it Borel reduces to isomorphism of countable structures in a fixed countable language. It is \emph{universal analytic}, also called \emph{complete analytic}, if it is analytic and every analytic equivalence relation on a standard Borel space Borel reduces to it. These notions belong to the framework developed by Friedman and Stanley \cite{FS} and Louveau and Rosendal \cite{LR}; see also \cite{Gao,Kechris}.

Write $G\topiso H$ for topological group isomorphism, and let $\pciso$ and $\naiso$ be its restrictions to $\PC$ and $\Subg(S_\infty)$, respectively. The analytic upper bounds are known.

\begin{proposition}\label{prop:ambient}
The class $\PC$ is Borel in $\Subg(S_\infty)$, and $\naiso$ is analytic. Consequently, $\pciso$ is analytic.
\end{proposition}
\begin{proof}
The first assertion is Proposition~3.3 of Gao, Nies, and Paolini \cite{GNP}; the second is recorded by Kechris, Nies, and Tent \cite[p.~1191]{KNT}. Restricting the analytic relation $\naiso$ to the Borel set $\PC$ proves the last assertion.
\end{proof}

\begin{theorem}\label{thm:main}
The relations $\pciso$ and $\naiso$ are universal analytic.
\end{theorem}

The conclusion already holds when every stage group is fixed. For a countable discrete group $H$, let $\End(H)$ and $\Epi(H)$ denote its endomorphisms and surjective endomorphisms, respectively, with the pointwise-convergence topology inherited from $H^H$. The homomorphism laws define a closed subspace $\End(H)$, and surjectivity is the intersection, over $h\in H$, of the open conditions $\exists g\in H\ (f(g)=h)$. Hence $\Epi(H)$ and $\mathcal E_H=\Epi(H)^{\N_{>0}}$ are Polish. For $\alpha=(\alpha_n)\in\mathcal E_H$, put
\begin{equation}\label{eq:fixed-group-limit}
 G_H(\alpha)=
 \bigl\{(h_n)\in H^{\N_{>0}}:\alpha_n(h_{n+1})=h_n
 \text{ for every }n\geq1\bigr\},
\end{equation}
with coordinatewise operations and the subspace topology. Define
\begin{equation}\label{eq:fixed-group-relation}
 \alpha\sim_H\beta\quad\Longleftrightarrow\quad
 G_H(\alpha)\topiso G_H(\beta)
 \qquad(\alpha,\beta\in\mathcal E_H).
\end{equation}

\begin{theorem}\label{thm:fixed-group}
There is a countable group $H$ such that $\sim_H$ on $\mathcal E_H$ is universal analytic.
\end{theorem}

The history of the problem involves two notions of universality. Friedman and Stanley showed that isomorphism of countable graphs and of countable groups is universal for isomorphism relations of countable structures \cite[Sections~1 and~2.3]{FS}. Their group construction uses Mekler's encoding by nilpotent groups of class two and fixed odd-prime exponent \cite[Section~2]{Mekler}. Louveau and Rosendal developed completeness methods for the larger class of analytic equivalence relations \cite{LR}. Ferenczi, Louveau, and Rosendal subsequently proved universality in this sense for permutative equivalence of unconditional basic sequences, uniform homeomorphism of complete separable metric spaces, and topological isomorphism of Polish groups, already on abelian groups \cite[Theorems~15 and~22 and Corollary~34]{FLR}.

Within the non-Archimedean class, local compactness leads to a more concrete classification. Kechris, Nies, and Tent proved that topological isomorphism of profinite groups and of locally compact non-Archimedean Polish groups is Borel bireducible with countable-graph isomorphism. Their general criterion gives the upper bounds; the lower bounds use discrete groups in the locally compact case and nilpotent profinite groups of class two and fixed odd-prime exponent in the profinite case \cite[Theorem~3.1, Section~4, and Theorem~4.3]{KNT}. Independently, Rosendal and Zielinski obtained the countable-structure upper bound for locally compact non-Archimedean Polish groups through compactifications \cite[Proposition~10]{RZ}. The exact complexity for arbitrary closed subgroups of $S_\infty$ remained open \cite[p.~1191]{KNT}.

Gao, Nies, and Paolini proved that procountable groups are not classifiable by countable structures. On $\R^\N$, define
\begin{equation}\label{eq:ell-infty-definition}
 x\mathrel{\ell_\infty}y
 \quad\Longleftrightarrow\quad
 \sup_{n\in\N}|x(n)-y(n)|<\infty
 \qquad(x,y\in\R^\N).
\end{equation}
They proved $\ell_\infty\leq_B\pciso$ and showed that $\pciso$ does not reduce to any orbit equivalence relation of a Borel action of a Polish group \cite[Theorem~1.2]{GNP}. For each odd prime $p$, their reduction takes values in procountable groups of nilpotency class two and exponent $p$ \cite[Theorem~2.4, Section~5, and Theorem~5.7]{GNP}. Its two stages pass through uniform homeomorphism of tree spaces and then encode these spaces by inverse systems of groups. They asked whether $\pciso$ is universal analytic \cite[Question~1.5]{GNP}. Theorem~\ref{thm:main} answers this question and determines the complexity of the full non-Archimedean class as well.

Gao, Nies, and Paolini's construction uses repeated Cantor blocks and a finite interval of target depths to recover bounded differences. Their spaces have a common homeomorphism type, and their inverse systems can be chosen with isomorphic stage groups \cite[Introduction and Sections~2.2--2.3]{GNP}. We retain these features while incorporating the permitted permutations of the weights. This is the additional information needed to pass from a bounded-difference lower bound to analytic universality.

Our source is permutative equivalence: two basic sequences are permutatively equivalent if a permutation of one makes the coefficientwise linear map extend to an isomorphism of their closed spans. Logarithms of normalized norms turn multiplicative comparison into bounded differences of integer weights. The allowed permutation must respect all rational linear combinations and remain the same at every inverse-system level. Section~\ref{sec:weights} encodes the rational linear structure in a fixed directed graph, and Section~\ref{sec:trees} places a copy of that graph on the distinguished vertices at every level.

The reconstruction theorem combines this control of permutations with the counting argument of \cite[Section~2.3]{GNP}. The two moduli of a uniform homeomorphism confine the images of a source block to a common finite interval of target depths. Counting prefixes recovers a bound on all weight differences. Conversely, a bounded change is implemented by shifting bits between a finite label and a binary tail. Theorem~\ref{thm:weight-reconstruction} identifies bounded equivalence with isomorphism of the inverse graph systems and with uniform or bi-Lipschitz equivalence respecting the graph.

The transfer to groups preserves compatibility across levels. Prze\'zdziecki's functor recovers a unique directed multigraph map from each nontrivial group homomorphism modulo inner automorphisms of the target \cite[Theorem~5.13]{Prz}. Section~\ref{sec:groups} uses this uniqueness to recover commuting stage maps and inverse identities. All stage graphs from level two onward are isomorphic to one fixed graph. Section~\ref{sec:borel} transports their bonding maps to its image group $H$ and proves Theorem~\ref{thm:fixed-group}; a Borel realization of the limits as closed permutation groups then gives Theorem~\ref{thm:main}.

\Needspace{12\baselineskip}
\section{Weights and their symmetries}\label{sec:weights}

\subsection{A universal bounded-weight relation}
We work over the real scalars. A basic sequence is a Schauder basis for its closed linear span. Equip $C([0,1],\R)$ with the supremum norm, and let $\B$ consist of the sequences $b=(b_i)$ such that $\norm{b_i}=1$ for every $i\in\N$ and
\begin{equation}\label{eq:suppression}
 \norm{\sum_{i\in F}a_i b_i}\leq\norm{\sum_i a_i b_i}
\end{equation}
for every finitely supported real sequence $(a_i)$ and finite $F\subseteq\N$. Taking $F$ to be a singleton proves linear independence. The coordinate projections are consequently well defined and contractive on the algebraic span. The initial-coordinate projections extend to its closure and converge to the identity by density. Thus $\B$ consists exactly of the normalized $1$-suppression-unconditional basic sequences in $C([0,1],\R)$. It is Borel: by continuity, rational coefficients suffice to test \eqref{eq:suppression}.

For $b,b'\in\B$, write $b\approx_\perm b'$ if there are $\sigma\in S_\infty$ and $K\geq1$ such that
\begin{equation}\label{eq:basis-equivalence}
 K^{-1}\norm{\sum_i a_i b_i}
 \leq\norm{\sum_i a_i b'_{\sigma(i)}}
 \leq K\norm{\sum_i a_i b_i}
\end{equation}
for every finitely supported real sequence $(a_i)$.

\begin{proposition}\label{prop:flr-source}
Permutative equivalence on $\B$ is universal analytic.
\end{proposition}
\begin{proof}
Ferenczi, Louveau, and Rosendal proved that permutative equivalence of unconditional basic sequences is universal analytic \cite[Theorem~15]{FLR}. Their Theorem~11(ii) and the proof of Theorem~15 give, for every analytic equivalence relation, a Borel reduction to sequences $(e_t)_{t\in S}$ of canonical vectors in a fixed separable real Banach space $\mathcal T_2$. These vectors are normalized, and all finite-coordinate projections are contractive \cite[Section~6 and the proof of Theorem~15, p.~335]{FLR}. Thus their construction has suppression constant one.

Enumerate each input tree in the order inherited from a fixed enumeration of the ambient nodes. This is Borel, since identifying the $j$th selected node requires finitely many membership tests. The Banach--Mazur embedding theorem states that every separable real Banach space embeds linearly and isometrically into $C([0,1],\R)$ \cite[Section~3, p.~325]{FLR}. Fix such an embedding of $\mathcal T_2$. Applying it coordinatewise places the cited reductions in $\B$ and preserves permutative equivalence. Finally, $\approx_\perm$ is analytic: in \eqref{eq:basis-equivalence}, quantify over $\sigma\in S_\infty$ and an integer $K$, and test rational finitely supported sequences.
\end{proof}

Let $c_{00}(\Q)$ be the rational vector space of finitely supported rational sequences. Regard $\A=c_{00}(\Q)$ as a structure in the language of rational vector spaces expanded by two unary predicates, interpreted as the set $\{e_i:i\in\N\}$ of standard basis vectors and the nonzero set $V=\A\setminus\{0\}$, respectively. Every member of $\Aut(\A)$ is the rational-linear extension of a unique permutation of the standard basis. Define
\begin{equation}\label{eq:weight}
 w_b(v)=\left\lceil-\log_2\frac{\norm{\sum_i v_i b_i}}{\sum_i|v_i|}\right\rceil
 \qquad(v\in V).
\end{equation}
The ratio lies in $(0,1]$: linear independence makes the numerator positive, while the triangle inequality and normalization bound it by $\sum_i|v_i|$. Thus $w_b(v)\in\N$, and each coordinate in \eqref{eq:weight} is Borel in $b$.

For $w,w'\in\N^V$, set
\begin{equation}\label{eq:Ebd}
 w\mathrel{\Ebd}w'\quad\Longleftrightarrow\quad
 \exists\pi\in\Aut(\A)\ \exists C\in\N\ \forall v\in V\quad
 |w(v)-w'(\pi v)|\leq C.
\end{equation}
This is an equivalence relation: inverses preserve the same bound, and composition adds bounds. It is analytic, since $\Aut(\A)$ is a closed permutation group on the fixed countable set $\A$, and the quantified condition for a fixed $C$ is closed.

\begin{proposition}\label{prop:quantization}
The Borel map $b\mapsto w_b$ reduces $\approx_\perm$ to $\Ebd$. In particular, $\Ebd$ is universal analytic.
\end{proposition}
\begin{proof}
If \eqref{eq:basis-equivalence} holds, let $\pi$ be the rational-linear extension of $\sigma$. It preserves the denominator in \eqref{eq:weight}. Taking logarithms and ceilings gives
\begin{equation}\label{eq:quantization-bound}
 |w_b(v)-w_{b'}(\pi v)|\leq\lceil\log_2K\rceil+1
 \qquad(v\in V).
\end{equation}
Conversely, for $r>0$ and $q=\lceil-\log_2r\rceil$, we have
\begin{equation}\label{eq:dyadic}
 2^{-q}\leq r<2^{1-q}.
\end{equation}
If \eqref{eq:Ebd} holds with bound $C$, applying \eqref{eq:dyadic} to the two ratios in \eqref{eq:weight} gives \eqref{eq:basis-equivalence} for rational coefficients with $K=2^{C+1}$. Continuity on each finite-dimensional coordinate space gives the same inequalities for real coefficients. The induced linear map and its inverse therefore extend to the closed spans. Universality follows from that of $\approx_\perm$.
\end{proof}

\subsection{A fixed graph of the permitted permutations}
We next encode the rational structure in a fixed graph. This permits us to carry its symmetries through the later group construction.

\begin{lemma}\label{lem:fixed-graph}
There are a countable irreflexive directed graph $\C=(D,E)$, an $\Aut(\C)$-invariant subset $P\subseteq D$, a bijection $\iota:V\to P$, and an isomorphism $\Theta:\Aut(\C)\to\Aut(\A)$ such that
\begin{equation}\label{eq:fixed-graph-action}
 \gamma(\iota(v))=\iota(\Theta(\gamma)v)
 \qquad(\gamma\in\Aut(\C),\ v\in V).
\end{equation}
\end{lemma}
\begin{proof}
Replace each positive-arity operation of $\A$ by its graph relation and the constant $0$ by the unary relation naming it; retain the two unary predicates. For every element $a$, introduce a vertex $x_a$. For each true tuple $R(\bar a)$ of arity $k$, introduce a tuple vertex $r_{R,\bar a}$ and position vertices $p_{R,\bar a,j}$, $j<k$, with edges
\begin{equation}\label{eq:incidence-edges}
 r_{R,\bar a}\longrightarrow p_{R,\bar a,j}\longrightarrow x_{a_j}.
\end{equation}
All introduced vertices are distinct. Color the element vertices according as $a=0$ or $a\in V$, the tuple vertices by $R$, and the position vertices by $(R,j)$, using disjoint colors for all these types. Call this colored graph the core. Its color-preserving automorphisms are exactly the unique extensions of automorphisms of $\A$.

Choose distinct integers $L_c\geq3$ for the colors. At each core vertex $x$ of color $c$, attach a fresh directed cycle $C_x$ of length $L_c$ by an edge from $x$ to a specified vertex $\rho_x\in C_x$. The resulting graph $\C$ is countable and irreflexive. The core is acyclic, and no edge returns to it from a cycle. Hence its vertices are exactly those lying on no directed cycle. Moreover, $\rho_x$ is the unique vertex of $C_x$ receiving an edge from outside the cycle. An automorphism therefore recovers every attachment and color. Conversely, a color-preserving core automorphism extends uniquely, since a directed cycle has no nontrivial automorphism fixing a vertex.

Set $P=\{x_v:v\in V\}$ and $\iota(v)=x_v$. The color of these vertices makes $P$ invariant. Restriction to element vertices defines $\Theta$; unique extension proves that it is an isomorphism and gives \eqref{eq:fixed-graph-action}.
\end{proof}

Fix a bijection $i\mapsto d_i$ from $\N$ onto $D$. For $w\in\N^V$, define its extension $\widetilde w:D\to\N$ by
\begin{equation}\label{eq:weight-extension}
 \widetilde w(d)=
 \begin{cases}w(\iota^{-1}(d)),&d\in P,\\0,&d\notin P,\end{cases}
 \qquad(d\in D).
\end{equation}
By Lemma~\ref{lem:fixed-graph}, the permitted permutations on $P$ are exactly those induced by $\Aut(\A)$. Since both extensions vanish off the invariant set $P$,
\begin{equation}\label{eq:graph-weights}
 w\mathrel{\Ebd}w'\quad\Longleftrightarrow\quad
 \exists\pi\in\Aut(\C)\ \exists C\in\N\ \forall d\in D\quad
 |\widetilde w(d)-\widetilde w'(\pi d)|\leq C
 \qquad(w,w'\in\N^V).
\end{equation}
Below, the tree construction is defined for any weight $a:D\to\N$; we apply it to $a=\widetilde w$.

\section{Recovering weights from inverse graph systems}\label{sec:trees}

The construction in this section applies to any fixed countably infinite irreflexive directed graph $\C=(D,E)$. We will use the graph from Lemma~\ref{lem:fixed-graph}. The tree records the size of each weight, while copies of $\C$ record the permitted permutations. The repeated blocks and the reconstruction estimate adapt \cite[Sections~2.2--2.3]{GNP} to this setting.

\subsection{The weighted tree}
A tree on $\N$ is a nonempty subset of $\N^{<\N}$ closed under initial segments. For a finite or infinite sequence $x$, write $x\restr n$ for its initial segment of length $n$; if $s$ is finite and $t$ is finite or infinite, then $s^\frown t$ denotes their concatenation. A tree is \emph{pruned} if every node extends to an infinite branch. Its branch space is
\begin{equation}\label{eq:branch-space}
 [T]=\{x\in\N^\N:x\restr n\in T\text{ for every }n\in\N\}.
\end{equation}
We equip $[T]$ with the complete ultrametric
\begin{equation}\label{eq:branch-metric}
 d_T(x,y)=2^{-\nu(x,y)},\qquad
 \nu(x,y)=\min\{j:x(j)\ne y(j)\}\quad(x\ne y).
\end{equation}
We put $d_T(x,x)=0$ and, for $s\in T$, denote the cylinder determined by $s$ by
\begin{equation}\label{eq:cylinder}
 N_s=\{x\in[T]:s\text{ is an initial segment of }x\}.
\end{equation}

Write $2\N=\{2k:k\in\N\}$. We call the elements of $(2\N)^\N$ the even branches. Start with the even tree $(2\N)^{<\N}$ and, for each $i\in\N$, distinguish the branch
\begin{equation}\label{eq:distinguished-branches}
 z_{d_i}=(2i,0,0,\ldots).
\end{equation}
When $d=d_i$, we also write this branch as $z_d$.
Given $a\in\N^D$, at each depth $r\geq2$ along $z_d$ attach the block $B_a(d,r)$ consisting of the branches satisfying
\begin{equation}\label{eq:block}
 x\restr r=z_d\restr r,\qquad
 x(r)=2q+1\quad(q\in\N,\ q<2^{a(d)}),\qquad
 x(r+s+1)\in\{1,3\}\quad(s\in\N).
\end{equation}
Let $T_a$ consist of $(2\N)^{<\N}$ and all finite initial segments of branches in these blocks. It is pruned: an all-even node has an all-even extension, while a node whose first odd coordinate has already appeared can be extended using the binary tail. Every branch of $T_a$ is either even or belongs to the unique block $B_a(d,r)$ determined by its first odd coordinate $r$. For fixed $q$, the corresponding set of branches is a compact open cylinder homeomorphic to $2^\N$. If $L>r$, there are $2^{a(d)}$ choices for $q$ and $2^{L-r-1}$ choices for the intervening tail, so $B_a(d,r)$ has
\begin{equation}\label{eq:capacity}
 2^{a(d)}2^{L-r-1}
\end{equation}
distinct prefixes of length $L$.

\begin{lemma}\label{lem:attached}
The union $Q_a$ of all attached blocks is exactly the set of points of $[T_a]$ that have a compact open neighborhood. Moreover,
\begin{equation}\label{eq:attached-boundary}
 \overline{Q_a}\setminus Q_a=\{z_d:d\in D\}.
\end{equation}
\end{lemma}
\begin{proof}
Each fixed-$q$ attached cylinder is compact and open. If $x$ is an even branch, every cylinder $N_{x\restr n}$ contains the branches
\begin{equation}\label{eq:separated-even-branches}
 (x\restr n)^{\frown}(2k,0,0,\ldots)\qquad(k\in\N),
\end{equation}
which are pairwise at distance $2^{-n}$; hence no such cylinder is compact. Any compact open neighborhood of $x$ would contain such a cylinder as a closed subset, a contradiction.

Every neighborhood of $z_d$ meets blocks attached farther along it. Now let $x$ be an even branch different from all $z_d$, and let $j\geq1$ be the first index with $x(j)\ne0$. For $e\in D$ and $r\geq2$, a branch in $B_a(e,r)$ has an odd coordinate at $r$. If $r\leq j$, that coordinate is incompatible with $x(r)$; if $r>j$, its coordinate at $j$ is zero. Thus $N_{x\restr(j+1)}$ misses $Q_a$, proving \eqref{eq:attached-boundary}.
\end{proof}

\subsection{The inverse graph system}
We now organize the levels of $T_a$ into an inverse system. All systems below are indexed by $\N_{>0}$. For a category $\mathsf K$, an inverse system $X$ consists of objects $X_n$ and bonding maps $p^X_{m,n}:X_m\to X_n$ for $m\geq n$, with the usual identity and composition laws. A \emph{pre-morphism} from $X$ to $Y$ is a strictly increasing (hence cofinal) map $\phi:\N_{>0}\to\N_{>0}$ and maps $f_n:X_{\phi(n)}\to Y_n$ such that
\begin{equation}\label{eq:premorphism}
 p^Y_{m,n}\circ f_m=f_n\circ p^X_{\phi(m),\phi(n)}
 \qquad(m\geq n\geq1).
\end{equation}
Two pre-morphisms $(\phi,f)$ and $(\phi',f')$ are equivalent if, for each $n\geq1$, there is $r\geq\phi(n),\phi'(n)$ such that
\begin{equation}\label{eq:premorphism-equivalence}
 f_n\circ p^X_{r,\phi(n)}=f'_n\circ p^X_{r,\phi'(n)}.
\end{equation}
Their equivalence classes are the morphisms in $\Pro(\mathsf K)$. The composite of $(\phi,f):X\to Y$ and $(\psi,g):Y\to Z$ has index $\phi\circ\psi$ and stage maps $g_n\circ f_{\psi(n)}$. A merely increasing cofinal index map can be replaced by a strictly increasing majorant, composing the stage maps with bonding maps. This gives an equivalent pre-morphism and agrees with \cite[Definition~3.5 and Remark~3.6]{GNP}.

We write $\GrpEpi$ for countable groups and surjective homomorphisms. A directed multigraph, or $m$-graph, consists of a vertex set, an edge set, and source and target maps from edges to vertices; both sets are countable when the $m$-graph is countable. An $m$-graph map is a pair of maps preserving source and target. Let $\mGraph$ be this category, as in \cite[Section~2]{Prz}, and let $\mGraphEpi$ have its countable objects and the maps surjective on vertices and edges. Equivalently, $\mGraph$ is the category of set-valued diagrams of shape $\mathsf E\rightrightarrows\mathsf V$, where $\mathsf E$ and $\mathsf V$ denote the edge and vertex sorts. Its epimorphisms are therefore precisely the maps surjective on both components. We regard an ordinary directed graph, possibly with loops, as an $m$-graph with one edge for each related pair.

\begin{lemma}\label{lem:identity}
Let $X$ be an inverse system in a category $\mathsf K$ whose bonding maps are epimorphisms. A self-pre-morphism $(\phi,f)$ represents the identity if and only if $f_n=p^X_{\phi(n),n}$ for every $n$.
\end{lemma}
\begin{proof}
Strict increase on $\N_{>0}$ implies $\phi(n)\geq n$. Equality with the identity gives, for each $n$, some $m\geq\phi(n)$ with
\begin{equation}\label{eq:identity-cancellation}
 f_n\circ p^X_{m,\phi(n)}=p^X_{m,n}
 =p^X_{\phi(n),n}\circ p^X_{m,\phi(n)}.
\end{equation}
Cancel the epimorphism on the right. The converse follows directly from the definition.
\end{proof}

For $n\geq1$, let $\Gamma_n(a)$ have vertex set $T_a\cap\N^n$. The distinguished vertices $z_d\restr n$, $d\in D$, are pairwise distinct because their first coordinates are distinct. Put a loop at each such vertex, and copy the edges of $\C$ between these vertices:
\begin{equation}\label{eq:stage-edges}
 z_d\restr n\longrightarrow z_e\restr n\quad\Longleftrightarrow\quad dEe
 \qquad(d,e\in D,\ d\ne e).
\end{equation}
There are no further edges. For $m\geq n$, define $p^a_{m,n}$ on vertices by $p^a_{m,n}(s)=s\restr n$; send each loop, and each copied edge corresponding to $(d,e)$, to the edge with the same label at level $n$. These maps are surjective on vertices because $T_a$ is pruned, and on edges because every loop or copied edge has a lift at every higher level. This defines
\begin{equation}\label{eq:graph-system}
 \Gamma(a)=(\Gamma_n(a),p^a_{m,n})_{m\geq n\geq1}
\end{equation}
as an object of $\Pro(\mGraphEpi)$.

A \emph{uniform homeomorphism} is a bijection whose forward and inverse maps are uniformly continuous.
A homeomorphism $h$ between metric spaces $(X,d_X)$ and $(Y,d_Y)$ is $L$-\emph{bi-Lipschitz}, for $L\geq1$, if
\begin{equation}\label{eq:bilipschitz-definition}
 L^{-1}d_X(x,y)\leq d_Y(h(x),h(y))\leq Ld_X(x,y)
 \qquad(x,y\in X).
\end{equation}
The following theorem gives the precise information retained by the inverse systems.

\begin{theorem}\label{thm:weight-reconstruction}
For $a,a'\in\N^D$, the following conditions are equivalent.
\begin{enumerate}
\item There are $\pi\in\Aut(\C)$ and $M\in\N$ such that
\begin{equation}\label{eq:weight-reconstruction-bound}
 |a(d)-a'(\pi d)|\leq M\qquad(d\in D).
\end{equation}
\item $\Gamma(a)\cong\Gamma(a')$ in $\Pro(\mGraphEpi)$.
\item There are $\pi\in\Aut(\C)$ and a uniform homeomorphism $h:[T_a]\to[T_{a'}]$ satisfying
\begin{equation}\label{eq:weight-reconstruction-action}
 h(z_d)=z_{\pi d}\qquad(d\in D).
\end{equation}
\item There are $\pi\in\Aut(\C)$ and a bi-Lipschitz homeomorphism $h:[T_a]\to[T_{a'}]$ satisfying \eqref{eq:weight-reconstruction-action}.
\end{enumerate}
If $M$ witnesses \textup{(1)}, the homeomorphism in \textup{(4)} can be chosen $2^M$-bi-Lipschitz.
\end{theorem}

We prove the theorem in the next two subsections. The main estimate turns uniform continuity in both directions into a bound common to all weights.

\subsection{The reconstruction estimate}

\begin{lemma}\label{lem:uniform}
For $a,a'\in\N^D$, if $\Gamma(a)\cong\Gamma(a')$ in $\Pro(\mGraphEpi)$, there are $\pi\in\Aut(\C)$ and a uniform homeomorphism $h:[T_a]\to[T_{a'}]$ with $h(z_d)=z_{\pi d}$ for every $d\in D$.
\end{lemma}
\begin{proof}
Choose inverse pre-morphisms with stage maps
\begin{equation}\label{eq:graph-premorphisms}
 f_n:\Gamma_{\phi(n)}(a)\twoheadrightarrow\Gamma_n(a'),\qquad
 g_n:\Gamma_{\psi(n)}(a')\twoheadrightarrow\Gamma_n(a)
 \qquad(n\geq1).
\end{equation}
Only the distinguished vertices carry loops, and an $m$-graph map sends a loop to a loop. Hence
\begin{equation}\label{eq:distinguished-stage-action}
 f_n(z_d\restr\phi(n))=z_{\pi_n(d)}\restr n
 \qquad(d\in D)
\end{equation}
for some map $\pi_n:D\to D$. If $m\geq n$, the pre-morphism identity evaluated at $z_d\restr\phi(m)$ gives
\begin{equation}\label{eq:distinguished-stage-coherence}
 z_{\pi_m(d)}\restr n=z_{\pi_n(d)}\restr n.
\end{equation}
Since $n\geq1$ and distinguished branches have distinct first coordinates, $\pi_m(d)=\pi_n(d)$. Write the common map as $\pi$; similarly, the maps $g_n$ determine a map $\tau:D\to D$. By Lemma~\ref{lem:identity},
\begin{equation}\label{eq:pro-composites}
 g_n\circ f_{\psi(n)}=p^a_{\phi(\psi(n)),n},\qquad
 f_n\circ g_{\phi(n)}=p^{a'}_{\psi(\phi(n)),n}
 \qquad(n\geq1).
\end{equation}
Evaluating \eqref{eq:pro-composites} on distinguished vertices gives $\tau\circ\pi=\pi\circ\tau=\mathrm{id}_D$. Thus $\pi$ is bijective and $\tau=\pi^{-1}$. A copied edge from $d$ to $e$ with $d\ne e$ cannot be sent to a loop, because $\pi(d)\ne\pi(e)$. Hence both $\pi$ and its inverse preserve $E$, so $\pi\in\Aut(\C)$.

Define $h$ by
\begin{equation}\label{eq:h-prefix}
 h(x)\restr n=f_n(x\restr\phi(n))
 \qquad(x\in[T_a],\ n\geq1).
\end{equation}
The pre-morphism identity makes these prefixes coherent, defining a unique branch $h(x)\in[T_{a'}]$. Equation~\eqref{eq:h-prefix} proves uniform continuity. Define $k:[T_{a'}]\to[T_a]$ analogously from $(\psi,g)$. For each $x\in[T_a]$ and $n\geq1$, \eqref{eq:pro-composites} gives
\begin{equation}\label{eq:branch-inverse}
 k(h(x))\restr n
 =g_n\bigl(f_{\psi(n)}(x\restr\phi(\psi(n)))\bigr)
 =x\restr n.
\end{equation}
The other composite satisfies the corresponding identity. Thus $k=h^{-1}$, and both maps are uniformly continuous. The action on distinguished branches follows from the definition of $\pi$.
\end{proof}

For trees $T,T'$, a uniformly continuous map $u:[T]\to[T']$, and $k\in\N$, define
\begin{equation}\label{eq:modulus-definition}
 \Omega_u(k)=\min\bigl\{\ell\in\N:
 \forall x,y\in[T]\ \bigl(x\restr\ell=y\restr\ell
 \Longrightarrow u(x)\restr k=u(y)\restr k\bigr)\bigr\}.
\end{equation}
The set in \eqref{eq:modulus-definition} is nonempty by uniform continuity.

\begin{proposition}\label{prop:window}
Let $a,a'\in\N^D$ and let $\pi:D\to D$ be a bijection. If $h:[T_a]\to[T_{a'}]$ is a uniform homeomorphism satisfying $h(z_d)=z_{\pi d}$ for every $d\in D$, then
\begin{equation}\label{eq:bounded-weights}
 \sup_{d\in D}|a(d)-a'(\pi d)|<\infty.
\end{equation}
\end{proposition}
\begin{proof}
The property of having a compact open neighborhood is preserved by homeomorphisms, so Lemma~\ref{lem:attached} gives $h(Q_a)=Q_{a'}$. Fix $m\geq2$ and put
\begin{equation}\label{eq:moduli}
 n=\max\{2,\Omega_h(m)\},\qquad L=\Omega_{h^{-1}}(n+1).
\end{equation}
For each $d$, choose one branch $x_q$ in each of the $2^{a(d)}$ fixed-$q$ cylinders of $B_a(d,n)$. Since $x_q$ and $z_d$ agree on their first $n$ coordinates, $h(x_q)$ and $z_{\pi d}$ agree on their first $m$ coordinates. Also $h(x_q)\in Q_{a'}$, so it lies in a unique block attached along $z_{\pi d}$, at some depth $p_q\geq m$.

If $p_q\geq L$, the image would agree with $z_{\pi d}$ on the first $L$ coordinates. By \eqref{eq:moduli}, $x_q$ and $z_d$ would then have the same $(n+1)$-prefix, contradicting the odd value of $x_q(n)$. Hence $m\leq p_q<L$, and in particular $L>m$. Further, the $L$-prefixes of the images are distinct: equality of two such prefixes would force equality of the source $(n+1)$-prefixes.

All $2^{a(d)}$ distinct prefixes therefore lie in the blocks whose attachment depths belong to $[m,L)$. Formula~\eqref{eq:capacity} gives
\begin{equation}\label{eq:window}
 2^{a(d)}\leq\sum_{p=m}^{L-1}2^{a'(\pi d)}2^{L-p-1}
 =2^{a'(\pi d)}(2^{L-m}-1)<2^{a'(\pi d)+L-m}.
\end{equation}
Thus $a(d)-a'(\pi d)<L-m$, uniformly in $d$. Applying the same argument to $h^{-1}$ and $\pi^{-1}$ gives a constant $C'$ such that $a'(e)-a(\pi^{-1}e)\leq C'$ for every $e\in D$. Substituting $e=\pi d$ gives the opposite bound.
\end{proof}

Only one source depth is used in this estimate. The moduli of $h$ and $h^{-1}$ confine its images, for every $d$, to the same finite interval of target depths. Counting the available prefixes converts these two moduli into a common additive bound on the weights. No individual attachment depth needs to be preserved.

\subsection{The converse}
\begin{proposition}\label{prop:recut}
Let $a,a'\in\N^D$. Suppose $C\in\N$, $\pi\in\Aut(\C)$, and $|a(d)-a'(\pi d)|\leq C$ for all $d\in D$. There is a $2^C$-bi-Lipschitz homeomorphism $h:[T_a]\to[T_{a'}]$ with $h(z_d)=z_{\pi d}$. It induces inverse pre-morphisms, giving $\Gamma(a)\cong\Gamma(a')$ in $\Pro(\mGraphEpi)$.
\end{proposition}
\begin{proof}
Write $\pi(d_i)=d_{\widehat\pi(i)}$. Then $\widehat\pi\in S_\infty$. On the even branches, define
\begin{equation}\label{eq:base-relabeling}
 (2i,2u_1,2u_2,\ldots)\longmapsto
 (2\widehat\pi(i),2u_1,2u_2,\ldots).
\end{equation}
This sends $z_d$ to $z_{\pi d}$. In the block $B_a(d,r)$, put $j=a(d)$ and $k=a'(\pi d)$. We use an empty binary word when $j=0$ or $k=0$, and encode the tail values $1$ and $3$ by $0$ and $1$, respectively. Express the cylinder label $q<2^j$ as a binary word of length $j$, including leading zeros, and append the encoded tail. This gives an infinite binary sequence. Read its first $k$ bits as a new cylinder label and the remainder as the tail in $B_{a'}(\pi d,r)$. Reversing this operation recovers the source branch, so these maps together define a bijection $h$.

For two distinct branches in the same block, let $t\in\N$ be the index of their first differing bit in the binary sequence above, with bits indexed from $0$. In the $j$-cut representation their first disagreement is at coordinate
\begin{equation}\label{eq:first-disagreement}
 F_{j,r}(t)=
 \begin{cases}
 r,&t<j,\\
 r+1+t-j,&t\geq j.
 \end{cases}
\end{equation}
For all $j,k,r,t\in\N$,
\begin{equation}\label{eq:cut-depth-bound}
 |F_{j,r}(t)-F_{k,r}(t)|\leq |j-k|.
\end{equation}
Indeed, suppose $j\leq k$. If $t<j$, the difference is zero; if $j\leq t<k$, it is $1+t-j\leq k-j$; and if $t\geq k$, it is $k-j$. The case $k\leq j$ is symmetric.

Outside a common block, the first disagreement is unchanged. Indeed, different first coordinates remain different; even branches retain their later coordinates; branches attached at different depths disagree at the smaller depth; and an even branch and an attached branch disagree either at the attachment depth or earlier. Together with \eqref{eq:cut-depth-bound}, this proves
\begin{equation}\label{eq:bilipschitz-depth}
 |\nu(h(x),h(y))-\nu(x,y)|\leq C
 \qquad(x,y\in[T_a],\ x\ne y).
\end{equation}
In particular, $h$ is $2^C$-bi-Lipschitz.

Set $\phi(n)=\psi(n)=n+C$. By \eqref{eq:bilipschitz-depth}, agreement on the first $n+C$ coordinates implies agreement of the images on the first $n$ coordinates, for both $h$ and $h^{-1}$. For $s\in T_a\cap\N^{n+C}$, define
\begin{equation}\label{eq:stage-vertex-map}
 f_n(s)=h(x)\restr n\quad\text{for any branch }x\text{ extending }s.
\end{equation}
Thus \eqref{eq:stage-vertex-map} is well defined, and surjectivity of $h$ makes $f_n$ surjective on vertices. On distinguished vertices it acts by $\pi$. Send the loop labeled by $d$ to the loop labeled by $\pi d$, and the copied edge labeled by $(d,e)$ to that labeled by $(\pi d,\pi e)$. This edge map is surjective because $\pi$ is an automorphism. Prefix compatibility gives \eqref{eq:premorphism}. The inverse map gives a pre-morphism $(\psi,g)$ with $g_n:\Gamma_{n+C}(a')\to\Gamma_n(a)$. On vertices, evaluating on any extending branch and using $h^{-1}\circ h=\mathrm{id}$ and $h\circ h^{-1}=\mathrm{id}$ gives
\begin{equation}\label{eq:recut-composites}
 g_n\circ f_{n+C}=p^a_{n+2C,n},\qquad
 f_n\circ g_{n+C}=p^{a'}_{n+2C,n}.
\end{equation}
The same equalities hold on loops and copied edges because their labels are carried successively by $\pi$ and $\pi^{-1}$. Hence these are equalities of $m$-graph maps, and Lemma~\ref{lem:identity} shows that the two pre-morphisms are inverse.
\end{proof}

\begin{proof}[Proof of Theorem~\ref{thm:weight-reconstruction}]
Proposition~\ref{prop:recut} gives \textup{(1)}$\Rightarrow$\textup{(2)} and \textup{(1)}$\Rightarrow$\textup{(4)}, including the stated constant. Lemma~\ref{lem:uniform} gives \textup{(2)}$\Rightarrow$\textup{(3)}, and \textup{(4)}$\Rightarrow$\textup{(3)} is immediate. Proposition~\ref{prop:window} gives \textup{(3)}$\Rightarrow$\textup{(1)}.
\end{proof}

\begin{remark}\label{rem:common-topology}
All spaces $[T_a]$ have the same homeomorphism type, even when every even branch is fixed. To see this, partition each space into the clopen sets
\begin{equation}\label{eq:first-coordinate-components}
 U_i^a=\{x\in[T_a]:x(0)=2i\}\qquad(i\in\N).
\end{equation}
For arbitrary $a,a'$, apply the binary recoding above on $U_i^a$ with $\pi=\mathrm{id}$, leaving its even branches fixed. The resulting map $U_i^a\to U_i^{a'}$ is $2^{|a(d_i)-a'(d_i)|}$-bi-Lipschitz. These maps combine to a homeomorphism because the sets in \eqref{eq:first-coordinate-components} form a disjoint open cover. Propositions~\ref{prop:window} and~\ref{prop:recut} show that a uniform homeomorphism fixing every $z_d$ exists exactly when these weight differences are uniformly bounded. Thus the distinction detected by the theorem lies within a single homeomorphism type: a separate bound on each component suffices for continuity, whereas uniform continuity imposes one bound on all components.

For example, take $a(d_i)=0$ and $a'(d_i)=i$. The spaces are homeomorphic but not uniformly homeomorphic. By Lemma~\ref{lem:attached}, any homeomorphism permutes the distinguished branches. For every such permutation $\pi$, the values $a'(\pi d)$ remain unbounded, so Proposition~\ref{prop:window} rules out a uniform homeomorphism.
\end{remark}

Applying Theorem~\ref{thm:weight-reconstruction} to \eqref{eq:graph-weights} gives the classification needed below.
\begin{corollary}\label{thm:graph-reconstruction}
For $w,w'\in\N^V$,
\begin{equation}\label{eq:graph-classification}
 w\mathrel{\Ebd}w'\quad\Longleftrightarrow\quad
 \Gamma(\widetilde w)\cong\Gamma(\widetilde w')
 \quad\text{in }\Pro(\mGraphEpi).
\end{equation}
\end{corollary}

\subsection{One graph at every stage}
We equip the weight space $\N^D$ with its product topology. For a countable $m$-graph $\Lambda$, write $V(\Lambda)$ and $E(\Lambda)$ for its vertex and edge sets. Give $\End_{\mGraph}(\Lambda)$ the pointwise-convergence topology inherited from
\begin{equation}\label{eq:graph-endomorphism-topology}
 V(\Lambda)^{V(\Lambda)}\times E(\Lambda)^{E(\Lambda)},
\end{equation}
with both sets discrete. Let $\Epi_{\mGraph}(\Lambda)$ denote the subspace of endomorphisms surjective on vertices and edges.

\begin{lemma}\label{lem:fixed-stage-graph}
Put $\Delta=\Gamma_2(0)$, where $0$ is the zero weight. For every $a\in\N^D$ and $n\geq2$, there is an isomorphism $\theta_n^a:\Gamma_n(a)\to\Delta$ preserving all distinguished vertex and edge labels. These isomorphisms can be chosen so that each map
\begin{equation}\label{eq:normalized-graph-maps}
 a\longmapsto q_k^a
 :=\theta_{k+1}^a\circ p^a_{k+2,k+1}\circ(\theta_{k+2}^a)^{-1}
 \quad(k\geq1)
\end{equation}
is continuous from $\N^D$ to $\Epi_{\mGraph}(\Delta)$.
\end{lemma}
\begin{proof}
Attachments first appear at length three, so $\Gamma_2(a)=\Delta$ for every $a$. For $n\geq2$, the graph $\Gamma_n(a)$ consists of the distinguished copy of $\C$, with its loops, and countably infinitely many isolated vertices. The latter already include the even nodes $(0,2k,0,\ldots,0)$ for $k\geq1$. Fix an enumeration of $\N^n$ for each $n$. Map the $j$th isolated vertex of $\Gamma_n(a)$, in the induced order, to the $j$th isolated vertex of $\Delta$, and preserve the distinguished labels. This defines $\theta_n^a$.

For a fixed finite sequence, membership in $T_a$ depends on at most one coordinate of $a$ and is clopen. Selecting the $j$th isolated vertex, or determining a given isolated vertex's rank, uses finitely many such tests. Consequently, the evaluation maps
\begin{equation}\label{eq:normalization-evaluation}
 (a,s)\longmapsto\theta_n^a(s),\qquad
 (a,v)\longmapsto(\theta_n^a)^{-1}(v)
\end{equation}
are continuous on $\{(a,s):s\in T_a\cap\N^n\}$ and $\N^D\times V(\Delta)$, respectively; all vertex sets carry the discrete topology. Composing these maps with the prefix map proves continuity of every vertex evaluation in \eqref{eq:normalized-graph-maps}. The edge maps preserve the fixed labels. Surjectivity follows from that of the prefix maps.
\end{proof}

\section{From graph systems to topological groups}\label{sec:groups}

We transfer the reconstruction theorem in two steps. Prze\'zdziecki's functor preserves surjective maps and recovers graph maps uniquely from group homomorphisms modulo inner conjugacy. This recovers the inverse systems. Continuity then recovers their stage maps from an arbitrary isomorphism of the inverse-limit groups.

\subsection{The almost-full functor}
Let $\mathsf{Grp}$ denote the category of groups and homomorphisms. We use the following result of Prze\'zdziecki \cite[Theorem~5.13]{Prz}.

\begin{theorem}\label{thm:almost-full}
There is a faithful functor $F:\mGraph\to\mathsf{Grp}$ for which, for every pair of $m$-graphs $\Gamma,\Delta$, the natural map
\begin{equation}\label{eq:almost-full}
 \Hom_{\mGraph}(\Gamma,\Delta)\sqcup\{*\}
 \ \,\longrightarrow\ \,
 \Hom_{\mathsf{Grp}}(F\Gamma,F\Delta)/\Inn(F\Delta)
\end{equation}
is a bijection. It sends an $m$-graph map $f$ to the class of $Ff$ and $*$ to the class of the trivial homomorphism; $\Inn(F\Delta)$ acts by postcomposition.
\end{theorem}

Consequently, every nontrivial homomorphism $F\Gamma\to F\Delta$ has the form $c_a\circ Ff$, where $c_a(x)=axa^{-1}$ and the $m$-graph map $f$ is unique.

Fix Prze\'zdziecki's specific functor. For an $m$-graph $\Gamma$, the group $F\Gamma$ is the colimit of a diagram $G\Gamma$. This diagram has one global vertex group $M$; for each vertex $v$ of $\Gamma$, a vertex group $P_{0,v}$ joined to $M$ through an edge group $N_v$; and, for each edge $e:v\to w$, vertex groups $P_{1,e},\ldots,P_{4,e}$ and edge groups $N_{0,e},\ldots,N_{4,e}$ arranged as
\begin{equation}\label{eq:diagram-architecture}
 P_{0,v}-P_{1,e}-P_{2,e}-P_{3,e}-P_{4,e}-P_{0,w},
\end{equation}
with successive edge groups $N_{0,e},\ldots,N_{4,e}$. All group types and incidence monomorphisms are fixed finite data \cite[Section~4]{Prz}. The vertex groups embed in $F\Gamma$ \cite[Lemma~5.2]{Prz} and together generate it \cite[proof of Lemma~5.10]{Prz}. In particular, $M=F\varnothing$ is a nontrivial subgroup of every $F\Gamma$, and $F\Gamma$ is countable whenever $\Gamma$ is countable.

\begin{lemma}\label{lem:surjections}
An $m$-graph map $f:\Gamma\to\Delta$ is surjective on vertices and edges if and only if $Ff$ is a surjective group homomorphism.
\end{lemma}
\begin{proof}
The functor acts as the identity on $M$ and by the standard isomorphisms between copies of the same type on all other vertex groups. If $f$ is onto on vertices and edges, the image of $Ff$ therefore contains $M$ and every vertex group $P_{0,w}$ and $P_{i,e}$ of the target diagram. These generate $F\Delta$, so $Ff$ is onto.

Conversely, suppose $Ff$ is onto. If $u\circ f=v\circ f$, then $Fu\circ Ff=Fv\circ Ff$. Surjectivity gives $Fu=Fv$, and faithfulness gives $u=v$. Thus $f$ is an epimorphism in $\mGraph$, hence surjective on both components.
\end{proof}

\begin{lemma}\label{lem:continuous-F}
For every countable $m$-graph $\Lambda$, the map
\begin{equation}\label{eq:continuous-F}
 \End_{\mGraph}(\Lambda)\longrightarrow\End(F\Lambda),
 \qquad f\longmapsto Ff,
\end{equation}
is continuous for the pointwise-convergence topologies. It sends $\Epi_{\mGraph}(\Lambda)$ into $\Epi(F\Lambda)$.
\end{lemma}
\begin{proof}
The group $F\Lambda$ is generated by the vertex groups of its fixed diagram. On these groups, $Ff$ fixes $M$, sends $P_{0,v}$ to $P_{0,f(v)}$, and sends $P_{i,e}$ to $P_{i,f(e)}$ by the prescribed isomorphisms between copies of the same finite group. For any $h\in F\Lambda$, choose a finite word in these groups representing $h$. Its image depends only on the values of $f$ at the finitely many vertex and edge labels occurring in that word. Thus $f\mapsto Ff(h)$ is continuous for each $h$, proving \eqref{eq:continuous-F}. The last assertion is Lemma~\ref{lem:surjections}.
\end{proof}

\begin{theorem}\label{prop:pro-reflection}
Applying $F$ levelwise induces a functor $\Pro(\mGraphEpi)\to\Pro(\GrpEpi)$. For any two inverse systems $\Gamma,\Delta$ in $\mGraphEpi$,
\begin{equation}\label{eq:pro-object-isomorphism}
 \Gamma\cong\Delta\text{ in }\Pro(\mGraphEpi)
 \quad\Longleftrightarrow\quad
 F\Gamma\cong F\Delta\text{ in }\Pro(\GrpEpi).
\end{equation}
\end{theorem}
\begin{proof}
The induced functor and the forward implication follow from Lemma~\ref{lem:surjections} and functoriality. For the reverse implication, write $\Gamma=(\Gamma_n,p_{m,n})$ and $\Delta=(\Delta_n,q_{m,n})$. Assume the right-hand side of \eqref{eq:pro-object-isomorphism}. Choose mutually inverse pro-morphisms represented by $(\phi,\alpha)$ and $(\psi,\beta)$, with stage maps
\begin{equation}\label{eq:group-stage-maps}
 \alpha_n:F\Gamma_{\phi(n)}\twoheadrightarrow F\Delta_n,
 \qquad\beta_n:F\Delta_{\psi(n)}\twoheadrightarrow F\Gamma_n
 \qquad(n\geq1).
\end{equation}
Each stage map is nontrivial, since it is onto a group containing the nontrivial subgroup $M$. By \eqref{eq:almost-full}, there are $m$-graph maps $f_n:\Gamma_{\phi(n)}\to\Delta_n$ and $g_n:\Delta_{\psi(n)}\to\Gamma_n$, and elements $a_n\in F\Delta_n$ and $b_n\in F\Gamma_n$, such that
\begin{equation}\label{eq:group-stage-decomposition}
 \alpha_n=c_{a_n}\circ Ff_n,\qquad
 \beta_n=c_{b_n}\circ Fg_n
 \qquad(n\geq1).
\end{equation}
Since inner conjugation is an automorphism, $Ff_n$ and $Fg_n$ are surjective. Lemma~\ref{lem:surjections} therefore makes $f_n$ and $g_n$ surjective on vertices and edges.

For $m\geq n$, use the identity $Fq_{m,n}\circ c_{a_m}=c_{Fq_{m,n}(a_m)}\circ Fq_{m,n}$. The group commutation equation
\begin{equation}\label{eq:group-commutation}
 Fq_{m,n}\circ\alpha_m
 =\alpha_n\circ Fp_{\phi(m),\phi(n)}
 \qquad(m\geq n\geq1)
\end{equation}
becomes, after substituting \eqref{eq:group-stage-decomposition},
\begin{equation}\label{eq:expanded-group-commutation}
 c_{Fq_{m,n}(a_m)}\circ F(q_{m,n}\circ f_m)
 =c_{a_n}\circ F(f_n\circ p_{\phi(m),\phi(n)}).
\end{equation}
Both homomorphisms in \eqref{eq:expanded-group-commutation} are nontrivial: before conjugation, they restrict to the identity on the distinguished copy of $M$. Uniqueness of the $m$-graph map in Theorem~\ref{thm:almost-full} therefore yields
\begin{equation}\label{eq:graph-commutation}
 q_{m,n}\circ f_m=f_n\circ p_{\phi(m),\phi(n)}
 \qquad(m\geq n\geq1).
\end{equation}
Thus $(\phi,f)$ is an $m$-graph pre-morphism, as is $(\psi,g)$.

For $m$-graph maps $f:\Gamma\to\Delta$ and $g:\Delta\to\Xi$, and elements $a\in F\Delta$ and $b\in F\Xi$, the identity
\begin{equation}\label{eq:conjugation}
 (c_b\circ Fg)\circ(c_a\circ Ff)
 =c_{b\,Fg(a)}\circ F(g\circ f)
\end{equation}
shows that extracting the $m$-graph map respects composition. By Lemma~\ref{lem:identity}, the two composite stage maps satisfy
\begin{equation}\label{eq:group-pro-composites}
 \beta_n\circ\alpha_{\psi(n)}=Fp_{\phi(\psi(n)),n},\qquad
 \alpha_n\circ\beta_{\phi(n)}=Fq_{\psi(\phi(n)),n}
 \qquad(n\geq1).
\end{equation}
These maps are nontrivial, since the right-hand sides fix $M$. Uniqueness in Theorem~\ref{thm:almost-full}, together with \eqref{eq:conjugation}, gives
\begin{equation}\label{eq:graph-pro-composites}
 g_n\circ f_{\psi(n)}=p_{\phi(\psi(n)),n},\qquad
 f_n\circ g_{\phi(n)}=q_{\psi(\phi(n)),n}
 \qquad(n\geq1).
\end{equation}
Lemma~\ref{lem:identity} now shows that the two $m$-graph pre-morphisms are inverse.
\end{proof}

The uniqueness in Theorem~\ref{thm:almost-full} turns equality modulo inner conjugacy into exact commutation of graph maps. Extracting these maps removes the conjugating elements at each stage.

\subsection{Intrinsic reconstruction from the limit}
Gao, Nies, and Paolini proved the following inverse-limit criterion \cite[Lemma~3.7(ii)]{GNP}. We include the argument to make the recovery of the stage maps explicit.

\begin{lemma}\label{lem:limits}
Two inverse systems $K,L$ in $\GrpEpi$, indexed by $\N_{>0}$, have topologically isomorphic inverse limits if and only if they are isomorphic in $\Pro(\GrpEpi)$.
\end{lemma}
\begin{proof}
Write $K=(K_n,p^K_{m,n})$ and $L=(L_n,p^L_{m,n})$, with every stage carrying the discrete topology.
A pre-morphism $(\phi,f):K\to L$ induces a continuous homomorphism $U:\varprojlim K\to\varprojlim L$. Let $\kappa_i:\varprojlim K\to K_i$ and $\lambda_i:\varprojlim L\to L_i$ be the coordinate projections. Then $U$ is characterized by
\begin{equation}\label{eq:induced-limit-map}
 \lambda_n\circ U=f_n\circ\kappa_{\phi(n)}\qquad(n\geq1).
\end{equation}
Equation~\eqref{eq:premorphism} makes these coordinates coherent. Equivalent pre-morphisms induce the same map, and composition is preserved. Thus inverse pro-morphisms give a topological isomorphism.

Conversely, let $u:\varprojlim K\to\varprojlim L$ be a topological isomorphism. The projections $\kappa_n$ and $\lambda_n$ are onto because the bonding maps are onto. The descending kernels $\ker\kappa_i$ form a neighborhood basis at the identity. Continuity therefore allows a strictly increasing choice of $\phi(n)$ with $\ker\kappa_{\phi(n)}\subseteq\ker(\lambda_n\circ u)$. This gives unique homomorphisms $\alpha_n:K_{\phi(n)}\to L_n$ satisfying
\begin{equation}\label{eq:limit-factorization}
 \lambda_n\circ u=\alpha_n\circ\kappa_{\phi(n)}
 \qquad(n\geq1).
\end{equation}
Each $\alpha_n$ is onto because $\lambda_n\circ u$ is onto. For $m\geq n$, the factorization equations give
\begin{equation}\label{eq:limit-factor-square}
 \begin{aligned}
 p^L_{m,n}\alpha_m\kappa_{\phi(m)}
 &=p^L_{m,n}\lambda_m u
 =\lambda_nu
 =\alpha_n\kappa_{\phi(n)}\\
 &=\alpha_np^K_{\phi(m),\phi(n)}\kappa_{\phi(m)}.
 \end{aligned}
\end{equation}
Surjectivity of $\kappa_{\phi(m)}$ yields the pre-morphism identity. Applying the same construction to $u^{-1}$ gives a pre-morphism $(\psi,\beta):L\to K$.

For one composite, the factorization equations give
\begin{equation}\label{eq:limit-composite-cancellation}
 (\beta_n\circ\alpha_{\psi(n)})\circ\kappa_{\phi(\psi(n))}
 =\kappa_n
 =p^K_{\phi(\psi(n)),n}\circ\kappa_{\phi(\psi(n))}
 \qquad(n\geq1).
\end{equation}
Since $\kappa_{\phi(\psi(n))}$ is surjective, \eqref{eq:limit-composite-cancellation} yields $\beta_n\circ\alpha_{\psi(n)}=p^K_{\phi(\psi(n)),n}$. The analogous calculation for the other composite and Lemma~\ref{lem:identity} show that the pre-morphisms are inverse.
\end{proof}

Thus continuity recovers the coordinate factors of any topological isomorphism. The chosen inverse-system presentation is not part of the final invariant.

\section{A fixed stage group and Borel realization}\label{sec:borel}

\subsection{Continuous families of bonding maps}
We now put every stage on the same countable group. This gives a continuous realization of the weight relation for any fixed graph.

\begin{proposition}\label{prop:fixed-realization}
Let $\C=(D,E)$ be a countably infinite irreflexive directed graph, construct $\Gamma(a)$ as in Section~\ref{sec:trees}, and put $H=F\Delta$, where $\Delta=\Gamma_2(0)$. There is a continuous map $\alpha:\N^D\to\mathcal E_H$ such that
\begin{equation}\label{eq:fixed-stage-classification}
 \alpha(a)\sim_H\alpha(a')
 \quad\Longleftrightarrow\quad
 \exists\pi\in\Aut(\C)\quad
 \sup_{d\in D}|a(d)-a'(\pi d)|<\infty.
\end{equation}
\end{proposition}
\begin{proof}
Use Lemma~\ref{lem:fixed-stage-graph} to define
\begin{equation}\label{eq:normalized-group-maps}
 \alpha(a)=(Fq_k^a)_{k\geq1}\in\mathcal E_H.
\end{equation}
This map is continuous by Lemmas~\ref{lem:fixed-stage-graph} and~\ref{lem:continuous-F}; its coordinates are surjective by Lemma~\ref{lem:surjections}. The isomorphisms $F\theta_{k+1}^a$ identify this system with the original group system starting at level two. Hence
\begin{equation}\label{eq:normalized-limit-isomorphism}
 G_H(\alpha(a))
 \topiso\varprojlim_{n\geq2}F\Gamma_n(a)
 \topiso\varprojlim_{n\geq1}F\Gamma_n(a).
\end{equation}
Deleting the first coordinate gives the second identification: that coordinate is recovered continuously from the second by the first bonding homomorphism. Theorem~\ref{thm:weight-reconstruction}, Theorem~\ref{prop:pro-reflection}, and Lemma~\ref{lem:limits} now give \eqref{eq:fixed-stage-classification}.
\end{proof}

Thus the relation on weights is realized with a fixed group at every level. Its dependence on the weights is entirely in the bonding endomorphisms.

\subsection{Closed permutation groups}
We finish by realizing the limits in the standard parameter space. The following construction is the fixed-group form of \cite[Remark~3.2]{GNP}; we include the verification of Borelness. Recall that the Effros Borel structure on $\mathcal F(S_\infty)$ is generated by
\begin{equation}\label{eq:effros-generators}
 \{K\in\mathcal F(S_\infty):K\cap U\ne\varnothing\},
 \qquad U\subseteq S_\infty\text{ open};
\end{equation}
see \cite{Effros} and \cite[Section~12.C]{Kechris}.

\begin{proposition}\label{prop:borel-realization}
For every countable group $H$, there is a Borel map
$R_H:\mathcal E_H\to\PC$ satisfying
$R_H(\alpha)\topiso G_H(\alpha)$ for every $\alpha\in\mathcal E_H$.
Consequently, $\sim_H$ is analytic.
\end{proposition}
\begin{proof}
Fix a bijection between $\Omega=\N_{>0}\times H$ and $\N$, and use it to identify $\operatorname{Sym}(\Omega)$ with $S_\infty$. For $\alpha\in\mathcal E_H$, let $G_H(\alpha)$ act on $\Omega$ by
\begin{equation}\label{eq:left-regular-action}
 \lambda_\alpha(g)(i,h)=(i,g_i h)
 \qquad(g=(g_i)\in G_H(\alpha),\ i\geq1,\ h\in H).
\end{equation}
Write $D_i=\{i\}\times H$ and $\xi_i=(i,1_H)$, and let $R_H(\alpha)$ be the image group. The images of the $\xi_i$ determine $g$, so the action is faithful. If $\kappa_i:G_H(\alpha)\to H$ is the coordinate projection, every point of $D_i$ has stabilizer $\ker\kappa_i$ under this action. The coordinate kernels form a neighborhood basis at the identity. Finite intersections of point stabilizers therefore give precisely the inverse-limit topology, so $\lambda_\alpha$ is a topological isomorphism onto its image.

If $\lambda_\alpha(g^{(r)})$ converges in $S_\infty$, the images of $\xi_i$ show that $g_i^{(r)}$ is eventually constant for every $i$. These eventual values are coherent and induce the limit permutation. Thus $R_H(\alpha)$ is closed and belongs to $\PC$.

For $j\geq i\geq1$, put
\begin{equation}\label{eq:fixed-bonding-composites}
 P_{j,i}(\alpha)=
 \begin{cases}
 \alpha_i\circ\alpha_{i+1}\circ\cdots\circ\alpha_{j-1},&j>i,\\
 \mathrm{id}_H,&j=i.
 \end{cases}
\end{equation}
Let $s$ be a finite injective partial permutation of $\Omega$ and let $[s]$ be the basic open set of its extensions. A constraint $(i,h)\mapsto(j,h')$ requires $j=i$ and forces the multiplier $c_i=h'h^{-1}$. All constraints on $D_i$ must force the same $c_i$; if either requirement fails, $[s]$ misses every $R_H(\alpha)$. Otherwise let $I_s$ be the finite set of constrained indices. If $I_s\ne\varnothing$, put $m=\max I_s$. Then
\begin{equation}\label{eq:effros-test}
 [s]\cap R_H(\alpha)\ne\varnothing
 \quad\Longleftrightarrow\quad
 P_{m,i}(\alpha)(c_m)=c_i\quad\text{for every }i\in I_s.
\end{equation}
A coherent tuple gives these identities. Conversely, extend $c_m$ downward by the bonding maps and upward by successive lifts under the surjective maps $\alpha_j$. The resulting coherent tuple acts as prescribed by $s$. Each evaluation in \eqref{eq:effros-test} depends continuously on $\alpha$, so the displayed test is Borel. For empty $s$, the identity permutation suffices. The sets $[s]$ form a countable basis, proving Borelness of $R_H$ for \eqref{eq:effros-generators}. Finally, \eqref{eq:fixed-group-relation} and Proposition~\ref{prop:ambient} show that $\sim_H$ is analytic.
\end{proof}

\begin{proof}[Proof of Theorems~\ref{thm:main} and~\ref{thm:fixed-group}]
Choose $\C$ from Lemma~\ref{lem:fixed-graph} and $H$ and $\alpha$ from Proposition~\ref{prop:fixed-realization}. Propositions~\ref{prop:quantization} and~\ref{prop:fixed-realization}, together with \eqref{eq:graph-weights}, give
\begin{equation}\label{eq:final-reduction}
 b\approx_\perm b'
 \quad\Longleftrightarrow\quad w_b\mathrel{\Ebd}w_{b'}
 \quad\Longleftrightarrow\quad
 \alpha(\widetilde w_b)\sim_H\alpha(\widetilde w_{b'})
 \qquad(b,b'\in\B).
\end{equation}
The assignment $b\mapsto\alpha(\widetilde w_b)$ is Borel: the weights are Borel by \eqref{eq:weight}, extension by zero is continuous, and $\alpha$ is continuous. Proposition~\ref{prop:flr-source} supplies the universal analytic source, while Proposition~\ref{prop:borel-realization} supplies analyticity of $\sim_H$. This proves Theorem~\ref{thm:fixed-group}.

The map $R_H$ reduces $\sim_H$ to $\pciso$, so $\pciso$ is universal analytic by Proposition~\ref{prop:ambient}. The Borel inclusion $\PC\hookrightarrow\Subg(S_\infty)$ then gives the same conclusion for $\naiso$, proving Theorem~\ref{thm:main}.
\end{proof}

\section{Limitations}\label{sec:limitations}

The construction does not impose abelianness, nilpotency, or a fixed exponent on the resulting groups. Universality on these subclasses remains outside the present argument. Gao, Nies, and Paolini's $\ell_\infty$ lower bound holds for nilpotency class two and each fixed odd-prime exponent \cite[Theorem~2.4, Section~5, and Theorem~5.7]{GNP}. Extending that result to universal analytic complexity calls for an encoding that also retains the permitted permutations and the compatibility of stage maps.

By Lemma~\ref{lem:attached}, the topology of $[T_a]$ identifies the set of distinguished branches $\{z_d:d\in D\}$, but it does not recover the edge relation of $\C$ on that set. The copies of $\C$ in the graphs $\Gamma_n(a)$ force the induced permutation to lie in $\Aut(\C)$. Therefore the present construction does not settle whether uniform homeomorphism of Polish ultrametric spaces is universal analytic, a question posed by Gao, Nies, and Paolini \cite[Question~1.6]{GNP}. This suggests encoding the incidence relation directly in the ultrametric space while preserving the counting estimate \eqref{eq:window}.

\section*{Declaration on generative AI and AI-assisted technology}

We employed our developed Agent TARS to assist with research-related tasks, and GPT-6-Astra for language-editing support. Xinan~Dai independently carried out error correction, mathematical verification, and manuscript curation. The authors take full responsibility for the mathematical arguments, references, and final content of this article.

\end{document}